\documentclass{amsart}
\usepackage{amsmath,amssymb,amsxtra,amsthm,amscd,verbatim}
\usepackage{lmodern}
\usepackage[all,cmtip]{xy}
\usepackage{rotating}
\usepackage{microtype}
\usepackage{todonotes}

\usepackage{listings}
\theoremstyle{plain}
\newtheorem{Thm}{Theorem}[section]

\newtheorem{Cor}[Thm]{Corollary}
\newtheorem{Lem}[Thm]{Lemma}
\newtheorem{Prop}[Thm]{Proposition}

\theoremstyle{definition}
\newtheorem{Def}[Thm]{Definition}

\newtheorem{Ex}[Thm]{Example}
\newtheorem{Rem}[Thm]{Remark}

\theoremstyle{remark}
\newtheorem{Not}[Thm]{Notation}

\DeclareMathOperator{\fil}{fil}

\DeclareMathOperator{\Z}{\mathbf{Z}}

\DeclareMathOperator*{\holimprep}{holim}                       
\newcommand{\holim}[1]%
{\displaystyle\holimprep_{\substack{\leftarrow \pull - \\ #1}} \,
}

\newcommand{\pull}
{\!\!\! -\!\!\! -\!\!\! -\!\!\!}
\DeclareMathOperator*{\hocolimprep}{hocolim}
\newcommand{\hocolim}[1]%
{\hocolimprep_{\substack{- \pull \rightarrow \\ #1}} \, }
\DeclareMathOperator*{\colimprep}{colim}
\newcommand{\colim}[1]%
{\colimprep_{\substack{- \pull \rightarrow \\ #1}} \, }

\newif\ifShowLabels
\ShowLabelstrue
\newcommand{\TeXref}[1]{
\marginpar{\scriptsize \texttt{#1}}}

\newcommand{\SecRef}[2]{\section{#1}\label{S:#2}%
\ifShowLabels \TeXref{{S:#2}} \fi}

\newcommand{\refT}[1]{\textup{\ref{T:#1}}}
\newcommand{\refL}[1]{\textup{\ref{L:#1}}}

\newenvironment{ThmRef}[1]%
{ \begin{Thm} \label{T:#1}
\ifShowLabels \TeXref{T:#1} \fi }%
{ \end{Thm} }
\newenvironment{DefRef}[1]%
{ \begin{Def} \label{D:#1}
\ifShowLabels \TeXref{D:#1} \fi }%
{ \end{Def} }
\newenvironment{LemRef}[1]%
{ \begin{Lem} \label{L:#1}
\ifShowLabels \TeXref{L:#1} \fi }%
{ \end{Lem} }
\newenvironment{CorRef}[1]%
{ \begin{Cor} \label{C:#1}
\ifShowLabels \TeXref{C:#1} \fi }%
{ \end{Cor} }
{ \begin{Rem} \label{R:#1}
\ifShowLabels \TeXref{R:#1} \fi }%
{ \end{Rem} }
{ \begin{Prop} \label{P:#1}
\ifShowLabels \TeXref{P:#1} \fi }%
{ \end{Prop} }
\newenvironment{ExRef}[1]%
{ \begin{Ex} \label{E:#1}
\ifShowLabels \TeXref{E:#1} \fi }%
{ \end{Ex} }
{ \begin{Not} \label{N:#1}
\ifShowLabels \TeXref{N:#1} \fi }%
{ \end{Not} }

\newenvironment{ThmRefName}[2]%
{ \begin{Thm} [#2]\label{T:#1}
\ifShowLabels \TeXref{T:#1} \fi }%
{ \end{Thm} }
\newenvironment{DefRefName}[2]%
{ \begin{Def} [#2]\label{D:#1}
\ifShowLabels \TeXref{D:#1} \fi }%
{ \end{Def} }
{ \begin{Lem} [#2]\label{L:#1}
\ifShowLabels \TeXref{L:#1} \fi }%
{ \end{Lem} }
{ \begin{Cor} [#2]\label{C:#1}
\ifShowLabels \TeXref{C:#1} \fi }%
{ \end{Cor} }
{ \begin{Rem} [#2]\label{R:#1}
\ifShowLabels \TeXref{R:#1} \fi }%
{ \end{Rem} }
{ \begin{Prop} [#2]\label{P:#1}
\ifShowLabels \TeXref{P:#1} \fi }%
{ \end{Prop} }
{ \begin{Ex} [#2]\label{E:#1}
\ifShowLabels \TeXref{E:#1} \fi  }%
{ \end{Ex} }

\newcommand{\coloneq}{:=}
\newcommand{\calE}{\mathcal{E}}

\newcommand{\inv}{^{-1}}
\newcommand{\N}{\mathbb{N}}

\usepackage{hyperref}

\ShowLabelsfalse

\begin{document}

\title[Colimit theorems for coarse coherence]{Colimit theorems for coarse coherence with applications}
\author[Boris Goldfarb]{Boris Goldfarb}
\author[Jonathan L. Grossman]{Jonathan L. Grossman}
\address{Department of Mathematics and Statistics\\ SUNY\\ Albany\\ NY 12222}
\date{\today}

\begin{abstract}
    We establish two versions of a central theorem, the Family Colimit Theorem, for the coarse coherence property of metric spaces.  This is a coarse geometric property and so is well-defined for finitely generated groups with word metrics.  It is known that coarse coherence of the fundamental group has important implications for classification of high-dimensional manifolds. The Family Colimit Theorem is one of the permanence theorems that give structure to the class of coarsely coherent groups.  In fact, all known permanence theorems follow from the Family Colimit Theorem.  We also use this theorem to construct new groups from this class.
\end{abstract}

\maketitle

\section{Introduction}

Let $\Gamma$ be a finitely generated group.  
Coarse coherence is a property of the group viewed as a metric space with a word metric, defined in terms of algebraic properties invariant under quasi-isometries.
It has emerged in \cite{gCbG:04,bGjG:18} as a condition that guarantees an equivalence between the $K$-theory of a group ring $K(R[\Gamma])$ and its $G$-theory $G(R[\Gamma])$.
The $K$-theory is the central home for obstructions to a number of constructions in geometric topology and number theory.  We view $K(R[\Gamma])$ as a non-connective spectrum associated to a commutative ring $R$ and the group $\Gamma$ whose stable homotopy groups are the Quillen $K$-groups in non-negative dimensions and the negative $K$-groups of Bass in negative dimensions.  The corresponding $G$-theory, introduced in this generality in \cite{gCbG:16}, is well-known as a theory with better computational tools available compared to $K$-theory. 

The two spectra are related by the so-called Cartan map $K(R[\Gamma]) \to G(R[\Gamma])$.  It is this map that is shown in \cite[Theorem 4.10]{bGjG:18} to be an equivalence if the group $\Gamma$ is coarsely coherent and belongs to Kropholler's hierarchy LH$\mathcal{F}$, and if $R$ is a regular Noetherian ring of finite global dimension. The algebraic conditions on the ring are satisfied in the most important for applications cases of $R$ being either the ring of integers or a field.

In order to give precise definitions and state our theorems, we need to review some background from coarse geometry and controlled algebra.

A map between metric spaces $f \colon X \to  Y$ is called \textit{uniformly expansive} if there is a function $c \colon [0, \infty) \to [0, \infty)$ such that $d_Y (f(x_1), f(x_2)) \le c (d_X (x_1, x_2))$ for all pairs of points $x_1$, $x_2$ from $X$.  
Two functions $h_1$, $h_2 \colon X \to Y$ between metric space are \textit{close} if there is a constant $C \ge 0$ so that $d_Y (h_1 (x), h_2(x)) \le C$ for all choices of $x$ in $X$.
A function $k \colon X \to Y$ is a \textit{coarse equivalence} if it is uniformly expansive and there exists a uniformly expansive function $l \colon Y \to X$ so that the compositions $k \circ l$ and $l \circ k$ are close to the identity maps.  In this case, we say that $k$ and $l$ are coarse inverses to each other.  If both obey the inequalities above for a function $c$, we say that $c$ is a \textit{control function} for both.

An example of a coarse equivalence is the notion of quasi-isometry.  This is simply a coarse equivalence $k$ for which the uniformly expansive functions for $k$ and its coarse inverse can be chosen to be linear polynomials.  In geometric group theory, it is very useful that any two choices for a finite generating set of a group produce quasi-isometric word metrics.
A map $f$ is a \textit{coarse embedding} if $k$ is uniformly expansive and is a coarse equivalence onto its image.  

Next we review some notions from controlled algebra related to $R$-modules filtered by the subsets of a metric space.  They are equivalent to the definitions given in \cite[section 2]{bGjG:18} and \cite[section 3]{gCbG:19}.
For the purposes of this paper it will be convenient to restate the definitions in terms of elements.

Let $X$ be a proper metric space.  An $X$\textit{-filtered} $R$\textit{-module} is a module $F$ together with a covariant functor $\mathcal{P}(X) \to \mathcal{I}(F)$
from the set of subsets of $X$ to the family of $R$-submodules of $F$, both ordered by inclusion, such that
the value on $X$ is $F$.
It is convenient to think of $F$ as the functor above and use notation $F(S)$ for the value of the functor on a subset $S$.
We will assume $F$ is \textit{reduced} in the sense that $F(\emptyset)=0$.

An $R$-linear homomorphism $f \colon F \to G$ of $X$-filtered modules is \textit{boundedly controlled} if there is a fixed number $b \ge 0$ such that for all subsets $S$ of $X$, if $x$ is in $F(S)$, its image $f (x)$ is an element of $G (S [b])$.
It is called \textit{boundedly bicontrolled} if there exists a number $b \ge 0$ such that, in addition to the property above, for all subsets $S \subset X$, if $y$ is in the image of $f$ and also in $G (S)$, it is the image $f (x)$ is some element $x$ from $F (S)$.
In this case we will say that $f$ has
\textit{filtration degree} $b$ and write $\fil (f) \le b$.

There are a few properties of filtered modules that we want to consider.

\begin{DefRef}{HYUT}
Let $F$ be an $X$-filtered $R$-module.
\begin{itemize}
\item $F$ is called \textit{lean} or $D$-\textit{lean} if there is a number $D \ge 0$ such that for for every subset $S$ of $X$ and any element $u$ in $F(S)$,
we have a decomposition $u = \sum_{x \in S} u_x$, where each $u_x$ is from $F(x[D])$ and only finitely many $u_x$ are non-zero.
\item $F$ is called \textit{scattered} or $D'$-\textit{scattered} if there is a number $D' \ge 0$ such that for any element $u$ in $F$, we have $u = \sum_{x \in X} u_x$, where $u_x$ is from $F(x[D])$ and only finitely many $u_x$ are non-zero..
\item $F$ is called \textit{insular} or $d$-\textit{insular} if there is a
number $d \ge 0$ such that any element $u$ simultaneously in $F(S)$ and in $F(T)$ is in fact an element of $F(S[d] \cap T[d])$.
\item $F$ is \textit{locally finitely generated} if $F(S)$ is a finitely generated $R$-submodule for all bounded subsets $S$.
\end{itemize}
\end{DefRef}

We will assume that all filtered modules are locally finitely generated.

Clearly, being lean is a stronger property than being scattered.  We refer the reader to some interesting examples of non-projective lean, insular filtered modules in \cite[Example 4.2]{gCbG:16}.

\begin{DefRefName}{CCoherence}{Coarse Coherence}
A metric space $X$ is \textit{coarsely coherent} if in any exact sequence 
\[
0 \to E' \xrightarrow{\ f \ } E \xrightarrow{\ g \ } E'' \to 0
\]
of $X$-filtered $R$-modules where $f$ and $g$ are both bicontrolled maps, the combination of $E$ being lean and $E''$ being insular implies that $E'$ is necessarily scattered.  
\end{DefRefName}

There is the following relaxation of the coarse coherence property.

\begin{DefRefName}{WCCoherence}{Weak Coarse Coherence}
A metric space $X$ is \textit{weakly coarsely coherent} if in any exact sequence 
\[
0 \to E' \xrightarrow{\ f \ } E \xrightarrow{\ g \ } E'' \to 0
\]
of $X$-filtered $R$-modules where $f$ and $g$ are both bicontrolled maps, the combination of $E$ being both lean and insular and $E''$ being insular implies that $E'$ is necessarily scattered.  
\end{DefRefName}

The fact that groups of finite asymptotic dimension are coarsely coherent was shown in \cite{gCbG:04}.  This was generalised to groups of countable asymptotic dimension in \cite{bG:20}.
The class of coarsely coherent groups has lots of permanence properties, such as invariance under very general extensions, products, etc. established by the authors in \cite{bGjG:18}, so there is a lot of structure already known in this class of groups. 

In this paper we prove an additional permanence property, the Family Colimit Theorem. There are two versions of this theorem.  One is a genuine permanence theorem for coarse coherence which holds under some geometric assumptions on the metric space.  The other is a general theorem with a slightly weaker conclusion.  Even though the conclusion is weaker, it has the same algebraic consequences in $K$-theory and so is as useful in that respect.  

In order to state our theorems, we formulate a version of coarse coherence for metric families.  A \textit{metric family} $\{ X_{\alpha} \}$ is simply a collection of metric spaces $X_{\alpha}$. 

\begin{DefRefName}{CoarseCoherenceFam}{Coarse Coherence for Families}
A metric family $\{ X_{\alpha} \}$ is \textit{coarsely coherent} if for a collection of exact sequences 
\[
0 \to E'_{\alpha} \xrightarrow{\ f_{\alpha} \ } E_{\alpha} \xrightarrow{\ g_{\alpha} \ } E''_{\alpha} \to 0
\]
of $X_{\alpha}$-filtered $R$-modules where all $E_{\alpha}$ are $D$-lean, all $E''_{\alpha}$ are $d$-insular, all $f_{\alpha}$ and $g_{\alpha}$ are all $b$-bicontrolled maps for some fixed constants $D$, $d$, $b \ge 0$, it follows that all $E'_{\alpha}$ are $\partial$-\textit{scattered} for some uniform constant $\partial \ge 0$.
\end{DefRefName}

The assumption that a family is coarsely coherent is certainly stronger than each space in the family being coarsely coherent.

Recall that a diagram of objects in a category $\mathcal{C}$ is a small subcategory in $\mathcal{C}$.

\begin{ThmRefName}{MJAXC}{The Family Colimit Theorem}
Suppose $\{ X_{\alpha} \}$ is a collection of metric spaces that is coarsely coherent as a family and $\mathcal{D}$ is a diagram of metric spaces where the nodes are members of $\{ X_{\alpha} \}$ and all of the structure maps are isometric embeddings.  Let $X$ be the colimit of $\mathcal{D}$. Then 
\begin{enumerate}
\item $X$ is weakly coarsely coherent,
\item $X$ is coarsely coherent if there are distance-nonincreasing functions $\tau_{\alpha} \colon X \to X_{\alpha}$ such that $d (x,\tau(x)) \le d(x,X_{\alpha}) + \varepsilon_{\alpha}$ holds for all $x$ in $X$ and some number $\varepsilon_{\alpha} \ge 0$.  
\end{enumerate}
\end{ThmRefName}

We prove the theorem in section 2.
In section 3, we will apply the theorem to argue that the wreath product $\Z \wr \Z$ is coarsely coherent. In section 4, we collect some consequences of the Family Colimit Theorem. The theorem allows us to construct two new families of coarsely coherent groups. In addition, we include a summary statement of known permanence properties for coarse coherence. 

\section{Colimit Theorems}

The proof of the Family Colimit Theorem will require several facts about modules filtered over the given family of metric spaces.

For any filtered module $F$ and a choice of number $D \ge 0$, there is new filtration on $F$ given by the formula $\widetilde{F}_D (S) = \sum_{x \in S} F(x[D])$.  More generally, $\widetilde{M}_D (S) = \sum_{x \in S} M \cap F(x[D])$ gives an $X$-filtration for a submodule $M$ of $F$.

\begin{LemRef}{JASK2}
The filtration $\widetilde{M}_D$ is always lean, for any submodule $M$ and for any value of $D \ge 0$.  If $F$ itself is $D$-lean then there is an interleaving between $F$ and $\widetilde{F}_D$ in the following sense: for all subsets $S$, one has $\widetilde{F}_D (S) \subset F(S[D])$ and $F(S) \subset \widetilde{F}_D (S)$.
\end{LemRef}

\begin{proof}
The equality
\[
\widetilde{M}_D (S) = \sum_{x \in S} M \cap F(x[D]) = \sum_{x \in S} \widetilde{M}_D (x)
\]
shows that $\widetilde{M}_D$ is always $0$-lean.
The first interleaving inclusion follows from the inclusion $x[D] \subset S[D]$ for $x \in S$.  The other is the $D$-leanness property of $F$.
\end{proof}

Let $X'$ be a metric subspace of $X$ and let $F$ be an $X$-filtered module.  For any function $\tau \colon X \to X'$, we will assign an $X'$-filtration to $F$, which we now denote $F_{\tau}$ in order to distinguish from the $X$-filtration.  As modules, $F(X) = F_{\tau}(X')$.  The $X'$-filtration is given by the formula $F_{\tau} (S) = F (\tau^{-1} (S))$.

\begin{LemRef}{JASK11}
For any metric pair $X' \subset X$ there is a function $\tau$ such that whenever $F (X) = F(X'[B])$ for some $B \ge 0$ then 
\begin{enumerate}
\item $F_{\tau}$ is lean if $F$ is lean and insular, and
\item $F_{\tau}$ is insular if $F$ is insular.
\end{enumerate}
\end{LemRef}

\begin{proof}
Choose a function $\tau \colon X \to X'$ such that $d (x, \tau (x)) \le d(x, X') + \varepsilon$ for a uniform choice of a number $\varepsilon \ge 0$.  On $X'[B]$, $\tau$ is bounded by $B + \varepsilon$. 

Suppose $F$ is $D$-lean and $d$-insular.  From the definition, $F_{\tau} (S) = F (\tau^{-1} (S))$.  From the assumption on $F$, $F_{\tau} (S) \subset F(X'[B])$.  The insularity property of $F$ gives 
\[
F_{\tau} (S) \subset F(\tau^{-1} (S)[d] \cap X'[B+d]).  
\]
Every point $x$ in the set $W = \tau^{-1} (S)[d] \cap X'[B+d]$ is within $d$ from a point $\overline{x}$ in $\tau^{-1} (S) \cap X'[B+2d]$.  
This means every point $y$ in $x[D]$ is within $d+D$ from $\overline{x}$. 
Let $x' = \tau (\overline{x})$, which is a point in $S$.  We know that $d(x', \overline{x}) \le B + 2d + \varepsilon$. Now we can estimate 
\begin{align}
d(x',\tau(x)) &\le d(x', \overline{x}) + d(\overline{x}, y) + d(y, \tau(y)) \notag \\ &\le (B+2d + \varepsilon) + (d + D) + (B+d + D + \varepsilon). \notag
\end{align}
This means $y$ is contained in $\tau^{-1}(x'[2B+4d + 2D+2\varepsilon])$, so the metric ball $x[D]$ in $X$ is contained in $\tau^{-1}(x'[2B+4d + 2D+2\varepsilon])$.  So we finally have 
\begin{multline}
F_{\tau} (S) \subset F(\tau^{-1} (S)[d] \cap X'[B+d]) \\ \subset \sum_{x \in W} F(x[D]) \subset \sum_{x' \in S} F_{\tau} (x'[2B+4d + 2D+2\varepsilon]). \notag
\end{multline}
This shows $F_{\tau}$ is $(2B+4d + 2D+2\varepsilon)$-lean.  

Insularity follows from a similar estimate.  Suppose $S$ and $T$ are subsets of $X'$.  Then using $d$-insularity of $F$, we have $F_{\tau} (S) \cap F_{\tau} (T) \subset F(\tau^{-1} (S)[d] \cap \tau^{-1} (T)[d])$.  We also assume $F \subset F(X' [B])$, so 
\begin{equation}
F_{\tau} (S) \cap F_{\tau} (T) \subset F(\tau^{-1} (S)[2d] \cap \tau^{-1} (T)[2d] \cap X'[B+d]).
\tag{$\ast$}
\end{equation}
Take a point $x$ in $\tau^{-1} (S)[2d] \cap \tau^{-1} (T)[2d] \cap X'[B+d]$, then there is $\overline{x}$ in $\tau^{-1} (S)$ such that $d(x,\overline{x}) \le 2d$.  We know that $d(x, \tau(x)) \le B+d+\varepsilon$, because $x$ is in $X'[B+d]$, and $d(\tau(\overline{x}),\overline{x}) \le (B+d)+2d+\varepsilon=B+3d+\varepsilon$. Now 
\begin{align}
d(\tau(\overline{x}),\tau(x)) &\le d(\tau(\overline{x}), \overline{x}) + d(\overline{x}, x) + d(x, \tau(x)) \notag \\ &\le (B+3d+\varepsilon) + 2d + (B+d+\varepsilon) = 2B + 6d+2\varepsilon. \notag
\end{align}
Let us use $d'$ for this constant $2B + 6d+2\varepsilon$.  The last inequality means that $x$ is in $\tau^{-1} (S[d'] \cap X')$ because $\tau(\overline{x})$ is in $S$.  Using the same estimates applied with respect to $T$ rather than $S$, we have that $x$ is also in $\tau^{-1} (T[d'] \cap X')$, and so in $\tau^{-1} (S[d'] \cap T[d'] \cap X')$.
This means that from $(\ast)$ above 
\begin{equation}
F_{\tau} (S) \cap F_{\tau} (T) \subset F_{\tau} (S[d'] \cap T[d'] \cap X').
\notag
\end{equation}
We conclude that $F_{\tau}$ is $(2B + 6d+2\varepsilon)$-insular. 
\end{proof}

In a special geometric situation where there is a distance-nonincreasing function $\tau \colon X \to X'$, stronger facts can be proved more easily.  Our main applications in this paper will in fact be made in this kind of situation.

\begin{LemRef}{JASK22}
Suppose $X' \subset X$ is a metric pair with a distance-nonincreasing function $\tau \colon X \to X'$. 
Then  $F_{\tau}$ is lean if $F$ is lean and is insular if $F$ is insular.
\end{LemRef}

\begin{proof}
The property of $\tau$ guarantees that for any $x$ in $X$ and any $D \ge 0$, $x[D] \subset \tau^{-1} (\tau(x)[D])$.  So $F(x[D])$ is contained in $F_{\tau} (\tau(x)[D])$.  Now given $S \subset X'$, 
\[
F_{\tau} (S) = F(\tau^{-1} (S)) \subset \sum_{x \in \tau^{-1} (S)} F(x[D]) \subset \sum_{x' \in S} F_{\tau} (x'[D]).
\]
This shows $F_{\tau}$ is $D$-lean if $F$ is $D$-lean.  A very similar argument gives that $F_{\tau}$ is $d$-insular if $F$ is $d$-insular.
\end{proof}

These lemmas will allow us to show Theorem \refT{MJAXC}.  As the reader must be expecting, the special case (2) will have a more direct proof than (1).

\begin{proof}[Proof of Theorem \refT{MJAXC}]
(1) Let 
\begin{equation}
0 \to E' \xrightarrow{\ f \ } E \xrightarrow{\ g \ } E'' \to 0 \tag{$\dagger$}
\end{equation}
be a short exact sequence of $X$-filtered $R$-modules such that both $f$ and $g$ are $b$-bicontrolled, $E$ is $D$-lean, and both $E$ and $E''$ are $d$-insular. We will demonstrate that $E'$ is scattered. 

Let $k$ be an element of $E'$. As $E$ is $D$-lean, and in particular $D$-scattered, $f(k)$ is a finite sum $\sum_{x} k_x$ for some elements $k_x \in E (x[D])$.  Notice that the elements $k_x$ are not necessarily in the kernel of $g$.  However, if we denote by $C$ the finite set of all $x$ with non-zero $k_x$, this establishes that $f(k) \in E(C[D])$.  Let $\alpha$ be an index such that $f(k)$ is an element of $E (X_{\alpha})$.  It is therefore an element of a larger submodule $\mathcal{E} = \widetilde{E}_D (X_{\alpha})$ defined in Lemma \refL{JASK2}, where it is given a $0$-lean $X$-filtration.  

We define a new $X$-filtered module $\mathcal{E}''$ as the image of $\widetilde{E}_D (X_{\alpha})$ under the homomorphism $g$ with the canonical filtration $\mathcal{E}'' (S) = \mathcal{E}'' \cap E''(S)$. It is immediate that $\mathcal{E}''$ is $d$-insular because $E''$ is $d$-insular.  It is also $b$-lean as an image of a $0$-lean module under a boundedly controlled homomorphism.  Note also that since $\widetilde{E}_D (X_{\alpha}) \subset {E} (X_{\alpha}[D])$, we have $\mathcal{E}'' \subset E'' (X_{\alpha}[D+b])$.
Choose a function $\tau \colon X \to X_{\alpha}$ such that $d (x, \tau (x)) \le d(x, X_{\alpha}) + \varepsilon$ for a fixed number $\varepsilon \ge 0$.  Define an $X_{\alpha}$-filtration of the $X$-filtered submodule $\mathcal{E}''$ according to the formula $E''_{\alpha} (S) = \mathcal{E}'' (\tau^{-1} (S))$.  The $X_{\alpha}$-filtered module $E''_{\alpha}$ is lean and insular by Lemma \refL{JASK11}.  Similarly, the $X_{\alpha}$-filtration $E_{\alpha} (S) = \mathcal{E} (\tau^{-1} (S))$ of $\mathcal{E}$ is lean and insular.
Now $k$ is in the kernel of a bicontolled epimorphism $g \vert \colon E_{\alpha} \to E''_{\alpha}$ from a lean $X_{\alpha}$-filtered module to an insular $X_{\alpha}$-filtered module. The assumption that $X_{\alpha}$ is coarsely coherent gives a decomposition $k = \sum_x k_{\alpha,x}$ in the kernel of $g \vert$.  The supports of $k_{\alpha,x}$ depend only on the constants involved and so independent of the choice of $k$ itself and from $\alpha$ because of the family condition.  This shows that $E'$ is scattered.

(2) Now suppose $X$ possesses distance-nonincreasing functions $\tau_{\alpha}$ as in the statement, and we consider an exact sequence ($\dagger$) where $E$ is lean and $E''$ is insular. The same constructions as in part (1) allow to construct an exact sequence of $X_{\alpha}$-modules.  This time $E_{\alpha}$ and $E''_{\alpha}$ are, respectively, lean and insular by applying Lemma \refL{JASK22}.  The conclusion is the same: $E'$ is scattered.
\end{proof}

\SecRef{Coarse Coherence of \texorpdfstring{$\Z\wr\Z$}{Z wreath Z}}{Example}

This example illustrates the use of Family Colimit Theorems.  It is then easily generalized to other groups with subexponential dimension growth \cite{aD:06} in the next section.

Recall that the restricted wreath product of finitely generated groups $G$ and $H$, denoted $G \wr H$, is the semi-direct product $\bigoplus_{h\in H} G \rtimes_\theta H$. There is a bijection between the set of elements of $\bigoplus_{h \in H}G$ and the set of functions $f \colon H \to G$ of finite support given by identifying any element $(g_1,g_2,\ldots)$ of $\bigoplus_{h \in H} G$ with the assignment function $f_g:H \to G$ sending $h_1 \mapsto g_1, h_2 \mapsto g_2, \ldots$. The twisting action of the semi-direct product may then be described by $(\theta_h(f_g))(h') = f_g(h\inv h')$. Let $g = (g_1, g_2, \ldots ) \in \bigoplus_{h \in H} G$ correspond to the assignment function $f_g$, as above. Then 
\begin{equation*}
    (g, h) \cdot (g', h') = (f_g \circ \theta_h(f_{g'}), hh')
    = (f_g \circ f_{g'}(h\inv \rule{6pt}{.5pt}), hh') 
    = (g \Tilde{g}', hh')
\end{equation*}
where $\Tilde{g}' = (\Tilde{g}_1', \Tilde{g}_2',\ldots)$ is the permutation of $g'$ given by taking each $g_i'=f_{g'}(h_i)$ and replacing it with $\Tilde{g}_i'= f_{g'}(h\inv h_i)$. The wreath product $\bigoplus_{h \in H} G \rtimes_\theta H$ is generated by $\Sigma_G \times \{1_H\} \cup \{1_G\} \times \Sigma_H$, where $\Sigma_G,\Sigma_H$ are finite generating sets for $G$ and $H$, respectively. The wreath product is therefore finitely generated. 

We refer to the word metric associated to the generating set $\Sigma_G \times \{1_H\} \cup \{1_G\} \times \Sigma_H$ as the \textit{wreath metric} and denote it $d_\wr$. We now set forth some facts regarding a specific wreath product that we propose to study: $\Z \wr \Z$.

Consider $\Z \wr \Z$ generated by $\{\pm1\} \times \{1\} \cup \{1\} \times \{\pm1\}$. Let $\gamma$ denote the generator $1_G$ from the first factor $\Z$ in the wreath product, and $\sigma$ denote the generator $1_H$ from the copy of $\Z$ that is the second factor in the wreath product. 

Observe that
\begin{equation*}
    X_n \coloneq \langle\gamma^{k}\sigma\gamma^{-k} \vert k = 1, \ldots, n\rangle \cong \Z^n \leq \Z \wr \Z
\end{equation*}
and that
\begin{equation*}
    X_\infty \coloneq \langle\gamma^{k}\sigma\gamma^{-k} \vert k \in \N\rangle \cong \Z^\infty \leq \Z \wr \Z.
\end{equation*}
If we use $e_i$ to denote the product $\gamma^i \sigma \gamma^{-i}$, we obtain the usual basis-like generators for $\Z^n$ and $\Z^\infty$. Denote this standard word metric on $\Z^n = \langle e_i \vert i = 1, \ldots, n\rangle$ by $d_n$ and denote the word metric on $\Z^\infty = \langle e_i \vert i \in \N \rangle$ by $d_\infty$. Each $e_i$ has length $2i+1$ in $(\Z \wr \Z,d_\wr)$, as is clear from their definition, and the distance between any pair $e_i, e_j$ is $2(i+j+1)$ from 
\begin{multline*}
\quad \quad d_\wr(e_i,e_j) = \|e_i\inv e_j\| = \|\gamma^{-i}\sigma\inv \gamma^i \gamma^j \sigma \gamma^{-j}\| \\= \|\gamma^{-i}\sigma\inv \gamma^{i+j} \sigma \gamma^{-j}\| = 2(i+j+1). \quad \quad 
\end{multline*}
The benefit of employing the wreath metric $d_\wr$ inherited from $\Z \wr \Z$ rather than the metric given by the infinite generating set $\{e_i\}_{i \in \N}$ is that $(\Z^\infty,d_\wr)$ is a subgroup of a finitely generated group $(\Z \wr \Z,d_\wr)$, and thus both $(\Z^\infty, d_\wr)$ and $(\Z \wr \Z, d_\wr)$  are discrete, proper, bounded geometry metric spaces. 

Further, the isomorphism $\phi \colon (X_n,d_\wr) \to (\Z^n,d_n)$ is a coarse equivalence, since the word metric on both spaces yields for any $\omega \in X_n$,
\begin{equation*}
    \|\omega\|_{\wr} = \|\phi(\omega)\|_{n} + 2n.
\end{equation*}
Denote the inverse isomorphism $\psi:(\Z^n,d_n) \to (X_n,d_\wr)$.

\begin{LemRef}{Z Infinity is CC}
The subgroup $(\Z^\infty,d_\wr)$ of $(\Z\wr\Z,d_\wr)$ is coarsely coherent.
\end{LemRef}

\begin{proof}
Let
\begin{equation*}
    0 \to E' \xrightarrow{f} E \xrightarrow{g} E'' \to 0 
\end{equation*}
be a short exact sequence of $\Z^\infty$-filtered $R$-modules with $E$ $D$-lean, $E''$ $d$-insular, and $f,g$ $b$-boundedly bicontrolled. Without loss of generality, assume that all constants are integers.

Observe that for any nonnegative integer $k$, there exists a nonnegative integer $n_k$ such that the ball of radius $k$ about the identity element of $(\Z^\infty,d_\wr)$ is contained in $(\Z^{n_k},d_\wr)$, which we will denote by $X$. The decomposition of $\Z^\infty$ into the cosets of $X$ is $k$-disjoint. Choose $k=2D+2b+2d$.  Since we can write $\Z^\infty$ as the union of $(2D+2b+2d)$-disjoint sets, it follows from the main theorem of \cite{bG:20} that
\begin{equation*}
    E' \subseteq \sum_{z \in \Z^\infty} E'((z\cdot X)[D])= \sum_{z \in \Z^\infty} E'(z[D]\cdot X).
\end{equation*}
The family $\{z \cdot X\}_z$ is isometric to $X$, and thus is a coarsely coherent family. Consequently, $\{z[D]\cdot X\}_z$ is a coarsely coherent family, as is $\{z[D+b]\cdot X\}_z$.

We present an argument analogous to that in \cite{bGjG:18}. Define $K_z = z[D+b] \cdot X$ and consider the $D$-leanly constructed $K_z$-filtered module
\begin{equation*}
    E_{K_z,D} 
    = \sum_{x \in K_z} E(x[D])
\end{equation*}
where for any $S \subseteq K_z$,
\begin{equation*}
    E_{K_z,D}(S) 
    = \sum_{x \in S \cap K_z} E(x[D]).
\end{equation*}
Denote this module by $\calE_z$, and the submodule associated to $S$ by $\calE_z(S)$, for brevity. By design, $\calE_z$ is lean, contains $f(E'(z[D]\cdot X))$. Define $\calE_z'' = g(\calE_z)$ equipped with the standard submodule filtration from $E''$, which is also a $K_z$-filtration, and define $\calE_z'= f\inv(\calE_z)$ equipped with the standard submodule filtration as well. The resulting short exact sequence of $K_z$-filtered modules
\begin{equation*}
    0 \to \calE_z' \xrightarrow{f\vert_{\calE_z'}} \calE_z \xrightarrow{g\vert_{\calE_z}} \calE_z'' \to 0
\end{equation*}
satisfies the requirements on the modules and module homomorphisms that coarse coherence entails. Thus, $\calE_z'$ is $\delta$-scattered for some constant $\delta \geq 0$. Therefore,
\begin{equation*}
     E' 
     \subseteq \sum_{z \in \Z^\infty} E'(z[D]\cdot X)
     \subseteq \sum_{z \in \Z^\infty} \sum_{x \in K_z} \calE_z'(x[\delta])
\end{equation*}
\begin{equation*}
     = \sum_{x \in \Z^\infty} E'(x[\delta]) \cap E'(K_z)
     \subseteq \sum_{x \in \Z^\infty} E'(x[\delta])
     ,
\end{equation*}
and $(\Z^\infty,d_\wr)$ is coarsely coherent.
\end{proof}

This argument utilizes that $(\Z^n,d_\wr)$ is coarsely coherent for all $n$, but does not require that they be so as a family. However, since $(\Z^\infty,d_\wr)$ is coarsely coherent, it then follows by the subspace permanence property of coarse coherence that the set $\{(\Z^n,d_\wr) \colon n \in \N\}$ is in fact a coarsely coherent family. 

Since the family $\{(\Z^n,d_\wr) \colon n \in \N\}$ is coarsely coherent, the family of subgroups $\{(\Z^n \rtimes \Z,d) \colon n \in \N\}$ is coarsely coherent (where $d$ is any metric that yields $d_\wr$ and $d_1$ when restricted to $\Z^n$ and $\Z$, respectively). From part (2) of Theorem \refT{MJAXC} we obtain the following result.

\begin{CorRef}{Z wreath Z is CC}
The wreath product $(\Z \wr \Z,d_\wr)$ is coarsely coherent. 
\end{CorRef}

\SecRef{Applications of the Colimit Theorem}{EA}

We view the Family Colimit Theorem as one from the collection of  permanence results for coarse coherence.  We will use other results of this type in our applications, so we want to state a summary of known permanence results.

In this section we consider generalized metric spaces where $\infty$ is a possible value of the metric.  In this setting one has metric disjoint unions.

Suppose $p \colon X \to Y$ is a uniformly expansive map of metric spaces.  Then $X$ is said to be a \textit{fibering} with base $Y$ and \textit{coarse fibers} which are preimages of metric balls in $Y$.

\begin{ThmRefName}{HGKAS}{Permanence Properties of Coarse Coherence}
The collection of coarsely coherent metric spaces is closed under the following operations:
\begin{enumerate}
    \item passage to metric subspaces,
    \item passage to commensurable metric overspaces,
    \item disjoint unions,
    \item uniformly expansive fiberings with the base and the coarse fibers in the collection.
\end{enumerate}

The subcollection of finitely generated groups with word metrics is closed under the following constructions:
\begin{enumerate} \setcounter{enumi}{4}
    \item passage to subgroups of finite index,
    \item passage to commensurable overgroups,
    \item finite semi-direct products, including finite direct products,
    \item free products with amalgamation and HNN extensions.
\end{enumerate}
\end{ThmRefName} 

\begin{proof}
Parts (5) and (6) follow from (1) and (2).
Part (4) is Theorem 2.12 of \cite{bGjG:18}. All of the remaining properties are proved in loc. cit. as a basis for stating a similar closure theorem for the coarse regular coherence property.  The proof of Theorem 4.12 of \cite{bGjG:18} gives precise references to results in that paper.
\end{proof}

For a countable discrete group with a fixed countable generating set, the \textit{word distance} $d(\gamma_1, \gamma_2)$ is defined as the minimal length of words in the given countable alphabet that represent $\gamma_1^{-1} \gamma_2$. 

Given a countable ordered generating set $A = \{ a_i \}$ for a group $\Gamma$, one may consider the sequence of subgroups $H_i$ generated by the prefix of size $i$ in $A$.  In other words, $H_i = \langle a_1, a_2, \ldots, a_i \rangle$.  One may also assume, without affecting the sequence of subgroups, that each $a_{t+1}$ is not contained in $H_t$.  This guarantees that each inclusion $H_i \to H_j$ for $i < j$ is a proper isometric embedding for the word metrics induced from the listed generating sets.

\begin{ThmRef}{VBNC}
If $\{ H_i \}$ is a coarsely coherent family, $\Gamma$ is weakly coarsely coherent.  If $\Gamma$ possesses distance-nonincreasing functions $\tau_{i} \colon \Gamma \to H_{i}$ such that $d (\gamma,\tau(\gamma)) \le d(\gamma,H_{i}) + \varepsilon_{i}$ for all $\gamma$ in $\Gamma$ and some $\varepsilon_{i} \ge 0$, then $\Gamma$ is coarsely coherent. 
\end{ThmRef}

\begin{proof}
This is precisely Theorem \refT{MJAXC} applied to the present situation.
\end{proof}

\begin{Cor} \label{NMKJ}
A countable discrete group is weakly coarsely coherent if and only if all of its finitely generated subgroups are coarsely coherent as a family. 
\end{Cor}

\begin{Cor} \label{QWSA}
If $G$ and $H$ are countable, discrete, coarsely coherent groups, then $G \wr H$, the reduced wreath product of $G$ with $H$, is coarsely coherent.
\end{Cor}

\begin{proof}
We can identify $G \wr H$ with the semi-direct product 
\begin{equation*}
\bigg(\bigoplus_{h \in H} G\bigg) \rtimes H = \bigg(\colim{n}  \bigoplus_{h \in H_n} G \bigg) \rtimes H,
\end{equation*}
where $H_n = \langle h_1,h_2,\ldots,h_n\rangle$, the group generated by the first $n$ elements of $H$ under any fixed ordering. From Corollary \ref{NMKJ}, 
$\bigoplus_{h \in H} G$
is coarsely coherent.
Since the collection of coarsely coherent groups is closed under semi-direct products, and $H$ is coarsely coherent, the result follows.
\end{proof}

\begin{ExRef}{GHJ}
Lamplighter groups $L_n$, for all natural numbers $n$, are the restricted wreath products given as $L_n = \Z_n \wr \Z$. Both $\Z_n$ and $\Z$ are coarsely coherent, so $L_n$ is coarsely coherent from Corollary \ref{QWSA}.
\end{ExRef}

\begin{Cor}
If $G$ is a countable, discrete, coarsely coherent group and $\alpha \colon H \to K$ is an isomorphism of subgroups of $G$, then the HNN extension $G\ast_\alpha$ is coarsely coherent.
\end{Cor}

\begin{proof}
HNN extensions are constructed by taking colimits, amalgamated free products, and semi-direct products with $\Z$, so the result follows from the preceding corollaries and Theorem \refT{HGKAS}.
\end{proof}

\end{document}